\documentclass[aap,preprint]{imsart}
\usepackage{amsfonts}
\RequirePackage[OT1]{fontenc}
\RequirePackage{amsthm,amsmath}
\RequirePackage[numbers]{natbib}
\RequirePackage[colorlinks,citecolor=blue,urlcolor=blue]{hyperref}
\usepackage{color}

\arxiv{arXiv:1511.00546}

\startlocaldefs
\numberwithin{equation}{section}
\theoremstyle{plain}

\endlocaldefs

\newcommand{\bigO}{\mathcal{O}}
\newcommand{\lr}[1]{\left(#1\right)}

\newcommand{\PHI}{\Phi^{(1)}}
\newcommand{\PHItwo}{\Phi^{(2)}}

\newcommand{\NUstar}{\nu^*}
\newcommand{\spm}{\{ 1, \ldots, q \}}
\newcommand{\Pois}[1]{\text{Poi}\lr{#1}}
\newcommand*{\BA}{\be \ba}
\newcommand*{\EA}{\ea \ee}

\newcommand*{\MUstar}{\mu^{*}}

\newcommand{\PHImin}{\phi_{\text{min}}}
\newcommand{\PHImax}{\phi_{\text{max}}}
\newcommand{\IND}[1]{1_{#1}}
\newcommand{\E}[1]{\mathbb{E}\left[#1\right]}

\newtheorem{theorem}{Theorem}[section]
\newcommand{\s}[2]{\sum_{#1}^{#2}}

\newtheorem{definition}[theorem]{Definition}
\renewcommand{\P}[1]{\mathbb{P}\left(#1\right)}
\newcommand{\set}[1]{\left\{#1\right\}}
\newcommand{\indicator}[1]{1_{\set{#1}}}

\newcommand*{\be}{\begin{equation}}
\newcommand*{\ee}{\end{equation}}
\newtheorem{conjecture}[theorem]{Conjecture}
\newcommand{\calT}{\mathcal{T}}
\newtheorem{proposition}[theorem]{Proposition}
\newtheorem{remark}[theorem]{Remark}

\newcommand{\calA}{\mathcal{A}}
\newcommand{\calE}{\mathcal{E}}
\newcommand{\calU}{\mathcal{U}}

\newtheorem{lemma}[theorem]{Lemma}
\newcommand*{\ba}{\begin{aligned}}
\newcommand*{\ea}{\end{aligned}}

\begin{document}

\begin{frontmatter}
\title{An Impossibility Result for Reconstruction in the Degree-Corrected Stochastic Block Model}
\runtitle{Impossibility Result for DC-SBM}
 
\begin{aug}
\author{\fnms{Lennart} \snm{Gulikers}\thanksref{m1,m3}\ead[label=e1]{lennart.gulikers@inria.fr}},
\author{\fnms{Marc} \snm{Lelarge}\thanksref{m2,m3}\ead[label=e2]{marc.lelarge@ens.fr}}
\and
\author{\fnms{Laurent} \snm{Massouli\'e}\thanksref{m1}
\ead[label=e3]{laurent.massoulie@inria.fr}
 }

\runauthor{L. Gulikers et al.}

\affiliation{Microsoft Research - INRIA Joint Centre\thanksmark{m1} and INRIA Paris\thanksmark{m2} and \'Ecole Normale Sup\'erieure\thanksmark{m3}}

\address{Postal address of L. Gulikers and L. Massouli\'e: \\
Microsoft Research-Inria Joint Centre,\\
 Campus de l'\'Ecole Polytechnique, 
 b\^atiment Alan Turing, \\
 1 rue Honor\'e d'Estienne d'Orves,
  91120 Palaiseau, France. \\
  E-mail addresses: \\
  lennart.gulikers@inria.fr and laurent.massoulie@inria.fr}

\address{Postal address of M. Lelarge:\\
Inria Paris,\\
 2 rue Simone Iff, CS 42112, 75589 Paris Cedex 12. \\
 E-mail address: \\ marc.lelarge@ens.fr}
\end{aug}

\begin{abstract}
We consider the Degree-Corrected Stochastic Block Model (DC-SBM): a random graph on $n$ nodes, having i.i.d. weights $(\phi_u)_{u=1}^n$ (possibly heavy-tailed), partitioned into $q \geq 2$ asymptotically equal-sized clusters. The model parameters are two constants $a,b > 0$ and the finite second moment of the weights $\PHItwo$. Vertices $u$ and $v$ are connected by an edge with probability $\frac{\phi_u \phi_v}{n}a$ when they are in the same class and with probability $\frac{\phi_u \phi_v}{n}b$ otherwise. 

We prove that it is information-theoretically impossible to estimate the clusters in a way positively correlated with the true community structure when  $(a-b)^2 \PHItwo \leq q(a+b)$. 

As by-products of our proof  we obtain $(1)$ a precise coupling result for local neighbourhoods in DC-SBM's, that we use in \cite{GuLeMa16} to establish a law of large numbers for local-functionals and $(2)$ that long-range interactions are weak in (power-law) DC-SBM's.
\end{abstract}



\end{frontmatter}

\section{Introduction}
It is well known that many networks exhibit a community structure. Think about groups of friends, web pages discussing related topics, or people speaking the same language (for instance, the Belgium population could be roughly divided into people speaking either Flemish or French). Finding those communities helps us understand and exploit general networks.

Instead of looking directly at real networks, we experiment first with models for networks with communities. One of the most elementary models is the Stochastic Block Model\footnote{SBM is standard terminology in the machine learning and statistics community, and is known as the Planted-Partition Model in theoretical computer science. The SBM is a special case of inhomogeneous random graphs, see \cite{BoBeJa07}.} (SBM) \cite{HoLaLe83}: a random graph on $n$ vertices partitioned into two equal-sized clusters such that vertices within the same cluster are connected with probability $p_{\text{in}}$ and  between the two communities with probability $p_{\text{out}}$. The question is now: given an instance of the SBM, can we retrieve the community membership of its vertices?

Most real networks are sparse and a thorough analysis of the sparse regime in the SBM - i.e., $p_{\text{in}} = \frac{a}{n}$ and $p_{\text{out}} = \frac{b}{n}$ for some constants $a,b>0$ - will therefore lead to a better understanding of networks. 

When the difference between $a$ and $b$ is small, the graph might not even contain enough information to distinguish between the two clusters. In \cite{DeAuKrFlMoCrZd11} it was first conjectured that a detectability phase-transition exists in the SBM: detection would be possible if and only if $\lr{a-b}^2 > 2(a+b)$. The negative side of this conjecture has been confirmed in \cite{MoNeSly15}. 
The positive side has been recently confirmed in \cite{Ma14} and \cite{MoNeSl13} using sophisticated (but still running in polynomial time) algorithms designed for this particular problem.

In this paper we study an extension of the SBM: a Degree-Corrected Stochastic Block Model (DC-SBM), see \cite{KaBrNe11}. 
Our motivation is as follows: although the SBM is a useful model due to its analytical tractability, it fails to accurately describe networks with a wide variety in their degree-sequences (because nodes in the same cluster are stochastically  indistinguishable). Indeed, real degree distributions often follow  a power-law \cite{AiChLu01}. Compare this to fitting a straight line on intrinsically curved data, which is doomed to miss important information.

The DC-SBM on $q$ communities is defined as follows: it is a random graph on $n$ vertices partitioned into $q$ asymptotically equal-sized clusters by giving each vertex $v$ a spin $\sigma_v$ drawn uniformly from $\spm$. The vertices have i.i.d. weights $\{\phi_u\}_{u=1}^n$ governed by some law $\nu$ with support in $W \subset [ \PHImin , \infty),$ where $0 < \PHImin < \infty$ is a constant \emph{independent of $n$}. We assume that the weights are possibly heavy-tailed with exponent $\beta > 8$: for all large enough $k$, \[\P{ \phi_1 \geq k } = \nu([k, \infty)) \leq \frac{1}{k^{\beta}}.\]   An edge is drawn between nodes $u$ and $v$ with probability $\frac{\phi_u \phi_v}{n} a$ when $u$ and $v$ have the same spin and with probability $\frac{\phi_u \phi_v}{n} b$ otherwise. The model parameters $a$ and $b$ are constant.

We denote the $k$-th moment of the weights by $\Phi^{(k)}$, i.e., $\Phi^{(k)} = \int_W x^k \mathrm{d} \nu(x)$.
We further introduce the following shorthand notation: $\sigma = (\sigma_1, \ldots, \sigma_n)$ and $\phi = (\phi_1, \ldots, \phi_n)$. For a subset $U \subset \{1,\ldots, n\}$ of the vertices, we define $\sigma_U = \{\sigma_u \}_{u \in U}$ and $\phi_U = \{\phi_u \}_{u \in U}$.

In the present paper we extend results in \cite{MoNeSly15} to the degree-corrected setting. More specifically, we prove that when $(a-b)^2 \PHItwo \leq q(a+b)$, it is information-theoretically impossible to estimate the spins in a way positively correlated with the true community structure based only on a single observation of the graph without knowing the weights. 

In a follow-up paper \cite{GuLeMa16}, we show that \textbf{in the two-community setting} above the threshold (i.e., $(a-b)^2 \PHItwo > 2(a+b)$), reconstruction is possible based on the second eigenvector of the so-called non-backtracking matrix. This is an extension of the results in \cite{BoLeMa15} for the \emph{ordinary} Stochastic Block Model.

We note that in \textbf{the two-community setting} there is an interpretation of the threshold  in terms of eigenvalues of the adjacency matrix $A$ given the weights. Indeed, if   $\psi_1$ and $\psi_2$ are the vectors defined for $u \in V$ by $\psi_1(u) = \frac{1}{\sqrt{2}} \phi_u$ and $\psi_2(u) = \frac{1}{\sqrt{2}} \sigma_u \phi_u$, then
$$
\E{A|\phi_1, \ldots, \phi_n} = \frac{a+b}{n}\psi_1 \psi_1^* + \frac{a-b}{n}\psi_2 \psi_2^* - a \frac{1}{n} \text{diag}\{\phi_u^2 \}.
$$
Thus, for $i=1,2$,  $\widehat{\psi}_i = \frac{\psi_i}{\| \psi_i \|_2}$ are the "mean-eigenvectors"  together with corresponding "mean-eigenvalues" $\lambda_1 = \frac{a+b}{2}\PHItwo$ and $\lambda_2 = \frac{a-b}{2}\PHItwo$: 
$$
\left\| \E{A|\phi_i, \ldots, \phi_n} \widehat{\psi}_i - \lambda_i \widehat{\psi}_i \right\|_2 \to 0,
$$
in probability, as $n$ tends to $\infty$.

We thus observe that the condition $(a-b)^2 \PHItwo \leq 2(a+b)$ is equivalent to $\lambda_2^2 \leq \lambda_1$.

\subsection{Our results}
\label{ssec::results}
In the sparse regime, $\Theta(n)$ vertices are isolated for which random guess is the only possible reconstruction-algorithm. In this paper, we therefore consider the community detection problem where we ask for a partition positively correlated with the true community structure:
\begin{definition}
Let $G$ be an observation of the DC-SBM, with true communities $\{\sigma_u\}_{u=1}^n$. Further, let $\{\widehat{\sigma}_u\}_{u=1}^n$ be a reconstruction of the communities, based on the observation $G$. 
Then, we say that $\{\widehat{\sigma}_u\}_{u=1}^n$ is \emph{positively correlated} with the true partition  $\{\sigma_u\}_{u=1}^n$ if there exists $\delta > 0$ such that 
\[ \P{\frac{1}{n} \s{u=1}{n} \indicator{ \sigma_u = \widehat{\sigma}_u} \geq \frac{1}{q} + \delta} \to 1, \]
as $n \to \infty$.
\label{def::pos}
\end{definition}

Our main result is:
\begin{theorem}
Assume that   $(a-b)^2 \PHItwo \leq q(a+b)$. Let  $G$ be an instance of the DC-SBM. Let $u$ and $v$ be uniformly chosen vertices in $G$.  Then, for any $s \in \spm$,
\be \P{\sigma_u = s | \sigma_v , G} \overset{\mathbb{P}} \to \frac{1}{q},\label{eq::con} \ee
as $n \to \infty$.
\label{thm::ConvergenceToHalf}
\end{theorem}
 Thus, it is already impossible to estimate the spin of a random vertex  given the spin of another vertex, which is an easier problem than reconstructing the group membership of strictly more than a fraction $1/q$ of the vertices (as explained in Lemma \ref{lm::NoConvergence}):
 
\begin{theorem}
Let $G$ be an observation of the DC-SBM with  $(a-b)^2 \Phi^{(2)} \leq  q(a+b)$. Then, no reconstruction $\{\widehat{\sigma}_u\}_{u=1}^n$ based on $G$ is positively correlated with $\{\sigma_u\}_{u=1}^n$.
\label{thm::Main}
\end{theorem}

As a by-product of our proof we obtain a precise coupling result for local neighbourhoods in DC-SBM's to weighted branching processes, such that the weights coincide exactly. This is an ingredient needed to prove a law of large numbers for local functionals that map neighbourhoods in the graph, together with their spins and weights to the real numbers. See Propositions $7.1$ and $7.2$ in \cite{GuLeMa16} for more details. Further, we also establish that long-range interactions are weak in DC-SBM's where the degrees follow a power-law with sufficiently large exponent. 

\subsection{General proof idea}
We first note that reconstruction is impossible when $\frac{a+(q-1)b}{q}\PHItwo \leq 1$, because in this regime there is no giant component\footnote{Indeed, the main result in \cite{BoBeJa07} concerns the existence, size and uniqueness of the giant component. In particular, in the setting considered here, a giant component  emerges if and only if $\frac{a+(q-1)b}{q} \PHItwo >1$. We shall henceforth assume a giant component to emerge.}. Note further that $\frac{a+(q-1)b}{q}\PHItwo \leq 1$ already implies $(a-b)^2 \PHItwo \leq q(a+b)$.

To establish \eqref{eq::con} when $\frac{a+(q-1)b}{q}\PHItwo > 1$ and $(a-b)^2 \PHItwo \leq q(a+b)$, 
we note that  $\text{Var}(\E{\sigma_u | \sigma_{\partial G_R} ,\sigma_v, G })$ is asymptotically an upper bound for $ \text{Var}( \E{\sigma_u | \sigma_v , G})$, as 
conditioning on the boundary spins $\sigma_{\partial G_R}$ of an $R$-neighbourhood around $u$ is more informative. 
 Now, we can approximate Var$(\E{\sigma_u  | \sigma_v , \sigma_{\partial G_R}, G}) \simeq $ Var$(\E{\sigma_u  |  \sigma_{\partial G_R}, G})$, because long-range correlations in this model are weak (Lemma \ref{lm::MarkovField}). Further, local neighbourhoods are w.h.p. tree-like, so that calculating the latter variance is equivalent to a certain tree-reconstruction problem  discussed in Section \ref{sec::broadcasting}. 
More specifically, we shall prove (Theorem \ref{thm::main_broadcasting}) that reconstruction of the spin of the root in a $q$-type tree (with offspring following a Poisson-mixture) based on the spins at depth $R$ (where $R \to \infty$), is impossible when $(a-b)^2 \PHItwo \leq q(a+b)$. Hence, $\text{Var}(\E{\sigma_u | \sigma_{\partial G_R} , G }) \to 0$ as $R \to \infty$.

\subsection{Background}
Without the degree correction (i.e., $\phi_1 = \ldots = \phi_n = 1$), the authors of \cite{DeAuKrFlMoCrZd11} were the first to conjecture  a phase-transition for the \emph{ordinary} SBM based on ideas from statistical physics:
\begin{conjecture}[\cite{DeAuKrFlMoCrZd11}]\label{conj}
Consider a SBM on $k$ balanced communities where edges inside a cluster are present with probability $a/n$ and between clusters with probability $b/n$. Let $M$ be the matrix with $a/n$ on the diagonal and $b/n$ on all off-diagonal elements. Let $\lambda_1$ and $\lambda_2$ be its first, respectively, second eigenvalue and let \emph{SNR} = $\frac{\lambda_2^2}{\lambda_1} = \frac{(a-b)^2}{k(a+(k-1)b)},$ the signal-to-noise-ratio. 

For any $k \geq 2$, if \emph{SNR} $ > 1$ (which is generally called the Kesten-Stigum condition), communities can be detected in polynomial time.

For $k \geq 4$, it is theoretically possible to detect communities for some \emph{SNR} $ < 1$.
\end{conjecture}

It is believed that for $k \geq 4$, a double phase-transition occurs: Detection should be easy (i.e., polynomial time) when \emph{SNR} $ > 1$, much harder (i.e., exponential time) for \emph{SNR} $\in (\tau,1]$, for some $0 < \tau < 1$, and information-theoretically impossible when \emph{SNR} $ < \tau$. 

The conjecture has been settled in the case of two communities: First in \cite{Ma14} by using a matrix counting the number of self-avoiding paths in the graph, and later, independently, in \cite{MoNeSl13}.
Further, \cite{MoNeSly15} shows that for $k=2$ , it is information-theoretically impossible to detect communities for \emph{SNR} below $1$. We shall here extend their results for the DC-SBM by relying on similar techniques.

In \cite{KrMoMoNeSlZdZh13} the 'spectral redemption conjecture' was made: detection using the second eigenvalue of the so called non-backtracking matrix would also establish the positive part. This has recently been proved\footnote{Theorems $4$ and $5$ in \cite{BoLeMa15} are actually a bit more general.} in \cite{BoLeMa15}, for any $k \geq 2$ such that $\lambda_k$ is a simple eigenvalue of $M$.

More recently, \cite{AbSa16} gave an algorithm that detects communities when\\ 
$\frac{(a-b)^2}{k(a+(k-1)b)} > 1$. 

Determining the 'hardness' of the intermediate regime (i.e., detection while below the Kesten-Stigum threshold) remains an open problem. \newline 
Positive results of spectral clustering in the DC-SBM have been obtained by various authors.
The work \cite{DaHoMcS04} introduces a reconstruction algorithm based on the matrix that is obtained by dividing each element of the adjacency matrix by the geometric mean of its row and column degrees.

 A slightly different extended stochastic block model is studied in \cite{CoLa08}: An edge is present between $u$ and $v$ with probability $\lr{1_{\{\sigma_u = \sigma_v\}}a + 1_{\{\sigma_u \neq \sigma_v\}}b } \cdot(\phi_u \phi_v) / ( \bar{\phi} n),$ where $\bar{\phi} = \frac{1}{n} \s{u=1}{n} \phi_u$, the average weight.
The main result is a polynomial time algorithm that outputs a partitioning that differs from the planted clusters on no more that $n \text{log}(\bar{\phi} ) / \bar{\phi}^{0.98}$ nodes. This recovery succeeds only under certain conditions: the minimum weight should be a fraction of the average weight and the degree of each vertex is $o(n)$. 

The article \cite{LeRi13} gives an algorithm based on the adjacency matrix of a graph together with performance guarantees. The average degree should be at least of order $\text{log} (n)$. However, since the spectrum of the adjacency matrix is dominated by the top eigenvalues \cite{ChLuVa03}, the algorithm does a poor job when the degree-sequence is very irregular. 

We propose in \cite{GuLeMa15} an algorithm that recovers consistently the block-membership of all but a vanishing fraction of nodes, even when the lowest degree is of order $\text{log} (n)$. It outperforms algorithms based on the adjacency matrix in case of heterogeneous degree-sequences.

\subsection{Outline and differences with \emph{ordinary} Stochastic Block model}
\label{ssec::outline}
We consider an associated tree reconstruction problem (see for instance \cite{EvKePeSc00,Mo04}) necessary for our analysis: given a tree, can we
deduce  the spin of the root based on all the spins at some distance $R \to \infty$ from the root?

We shall see that the $R$-neighbourhood of a vertex looks like a tree labelled with $q$ colors denoted here by $T^{\text{Poi}}$ and defined as follows.
We begin with a single particle, the root $o$, having spin $\sigma_o \in \spm$ and weight $\phi_o \in W \subset [\PHImin, \infty)$ (which we take random). The root is replaced in generation $1$ by $\Pois{\frac{a}{q} \PHI \phi_o}$ particles of spin $\sigma_o$ and  by $\Pois{\frac{b}{q} \PHI \phi_o}$ particles of spin  $s$ for each $s \in \spm \setminus \sigma_o$. Further, the weights of those particles are i.i.d. distributed following law $\NUstar$, the size-biased version of $\nu$, defined for $x \in [\PHImin, \infty)$ by
\be
\label{eq::size_biased}
\NUstar([0,x]) = \frac{1}{\PHI} \int_{\PHImin}^x y \mathrm{d} \nu(y).
\ee
 For generation $t \geq 1$, a particle with spin $\sigma$ and weight $\phi^*$ is replaced in the next generation by $\Pois{\frac{a}{2} \PHI \phi^*}$ particles with the same spin and $\Pois{\frac{b}{q} \PHI \phi_o}$ particles of each of the remaining $q-1$ spins. Again, the weights of the particles in generation $t+1$ follow in an i.i.d. fashion the law $\NUstar$. The offspring-size of an individual is thus a \textbf{Poisson-mixture} with mean $\frac{a+(q-1)b}{q}\PHItwo$.
 
  Section \ref{sec::broadcasting} deals with branching processes where the offspring is governed by a Poisson-mixture. The main theorem (i.e., Theorem \ref{thm::main_broadcasting}) deals with a reconstruction problem on these branching processes.

 In Section \ref{sec::coupling} we establish a coupling between the local neighbourhood and $T^{\text{Poi}}$. This result does not follow directly from the coupling in \cite{BoBeJa07}, because we need the weights in the graph and their counterparts in the branching process to be \emph{exactly} the same. 

Finally, in Section \ref{sec::weak} we show that long-range interactions are weak. The proof of Lemma \ref{lm::MarkovField} is  based on an idea in the proof of Lemma $4.7$ in \cite{MoNeSly15}. Note however that (besides the presence of weights) the statement of our Lemma \ref{lm::MarkovField} is slightly stronger than Lemma $4.7$ in \cite{MoNeSly15}, see below for details.

\section{Broadcasting on the branching process }
\label{sec::broadcasting}
Here we repeat without changes the definition of a Markov broadcasting process on trees given in \cite{EvKePeSc00, MoNeSly15}.
Let $\calT$ be an infinite tree with root $\rho$. Given a number $0 \leq \epsilon < 1/(q-1)$, define a random labelling $\tau \in \spm^{\calT}$ as follows: First, draw $\tau_{\rho}$ uniformly in $\spm$. Then, conditionally independently given $\tau_{\rho}$,  take every child $u$ of $\rho$ and, then with probability $1 - (q-1)\epsilon$ set $\tau_u = \tau_{\rho}$, and with probability $(q-1)\epsilon$ choose $\tau_u$ uniformly from $\spm \setminus \tau_{\rho}$  . Continue this construction recursively to obtain a labelling $\tau$ for which every vertex, independently, has probability $1-(q-1)\epsilon$ of having the same label as its parent and probability $\epsilon$ for each of the remaining spins. 

Suppose that the labels $\tau_{\partial \calT_m}$ at depth $m$ in the tree are known (here, $\tau_U = \{ \tau_i : i \in U \}$ and $\partial \calT_m$ are all vertices at distance $m$ from the root). The paper \cite{EvKePeSc00} gives precise conditions in the case of two spins as to when reconstruction of the root label is feasible using the optimal reconstruction strategy (maximum likelihood), i.e., deciding according to the sign of $\E{\tau_{\rho} |  \tau_{\partial \calT_m}}$.
Interestingly, this is completely decided by the branching number of $\mathcal{T}$ and the flip-probability $\epsilon$. The paper \cite{MoPe03} extends the results in \cite{EvKePeSc00} to the case of a general number of spins. For completeness we state both theorems here.
\begin{definition}
The branching number of a tree $\mathcal{T}$, denoted by $\emph{Br}(\calT)$, is defined as follows:
\begin{itemize}
\item If $\calT$ is finite, then $\emph{Br}(\calT) = 0$;
\item If $\calT$ is infinite, then we define the branching number in terms of percolation. Suppose that we retain each edge in the tree independently with probability $p$. Then $\emph{Br}(\calT)$ is the unique number such that: If $p <  \frac{1}{\emph{Br}(\calT)}$, then all components of the graph are finite a.s., while if $p >  \frac{1}{\emph{Br}(\calT)}$, then the graph has infinite components a.s.
\end{itemize}
\end{definition}

Remark that \cite{EvKePeSc00} does not deal with the trivial case of \emph{finite} trees. On such trees, Br$(\calT)=0$ by convention. This makes sense because, for large $m$, $\partial \calT_m = \emptyset$, and consequently  $\P{\tau_{\rho} = +|  \tau_{\partial \calT_m}} = 1/q.$ 

Theorem $1.1$ in \cite{EvKePeSc00} and Proposition $1.3$ in \cite{MoPe03} read, tailored to our needs:

\begin{theorem}(Theorem $1.1$ in \cite{EvKePeSc00})
For $q=2$, consider the problem of reconstructing $\tau_{\rho}$ from the spins $\tau_{\partial \calT_m}$ at the $m$th level of $\calT$. Define $\Delta_m$ as the difference between the probability of correct and incorrect reconstruction given the information at level $m$:
\[ \Delta_m:= \left| \P{\tau_{\rho} = +|  \tau_{\partial \calT_m}} -  \P{\tau_{\rho} = -|  \tau_{\partial \calT_m}} \right|. \]
If $\emph{Br}(\calT)(1 - 2 \epsilon)^{2} > 1$ then $\lim_{m \to \infty} \E{\Delta_m} > 0$.
\\ If, however, $\emph{Br}(\calT)(1 - 2 \epsilon)^{2} < 1$ then $\lim_{m \to \infty} \E{\Delta_m} = 0$.
\label{thm::Evans}
\end{theorem}

\begin{theorem}[Proposition $4.2$ in \cite{MoPe03}]
\label{thm::Evans_general}
For general $q \geq 2$, consider the problem of reconstructing $\tau_{\rho}$ from the spins $\tau_{\partial \calT_m}$ at the $m-$th level of $\calT$. Define $\textbf{P}_m^s$ as the conditional distribution of $\tau_{\partial \calT_m}$ given that $\sigma_{\rho} = s$. Then, $\lim_{m \to \infty} \|\textbf{P}_m^i - \textbf{P}_m^j \|_{\emph{TV}} = 0$ if $\emph{Br}(\calT)\frac{(1 - q \epsilon)^{2}}{1-(q-2)\epsilon} < 1.$ 
\end{theorem}
\begin{remark}
\label{rem::Evans_general}
Note that if $\emph{Br}(\calT)\frac{(1 - q \epsilon)^{2}}{1-(q-2)\epsilon} < 1,$ then
\BA & \E{ \left| \P{\tau_{\rho}=i|\tau_{\partial \calT_m}} - \P{\tau_{\rho}=j|\tau_{\partial \calT_m}} \right|} 
\\ &  = \sum_A \P{\tau_{\partial \calT_m} = A} \left| \P{\tau_{\rho}=i|\tau_{\partial \calT_m} = A} - \P{\tau_{\rho}=j|\tau_{\partial \calT_m} = A}  \right|
\\ &=  \frac{1}{q} \sum_A \left| \textbf{P}_m^i(A) - \textbf{P}_m^j(A) \right| \to 0,
\EA
as $m \to \infty$. Thus Theorem \ref{thm::Evans_general} implies Theorem \ref{thm::Evans}.
\end{remark}

 Note that in these theorems the tree is fixed, compared to the setting in this paper where  the multi-type branching process of Section \ref{ssec::outline} is considered.
But, it can be easily seen that the spins on a fixed instance $\calT$ of $T^{\text{Poi}}$ are distributed according to the above broadcasting process.

We thus need to calculate the branching number of a typical instance $\calT$: 
\begin{proposition}\label{prop::branching}
 Consider the multi-type branching process $T^{\emph{Poi}}$, where the root has spin drawn uniformly from $\spm$ and weight governed by $\nu$.  
Then, given the event that the branching process does not go extinct,  $\emph{Br}\lr{T^{\emph{Poi}}} \leq \frac{a+(q-1)b}{q} \PHItwo$ almost surely.
\end{proposition}
\begin{proof}
Denote the multi-type branching process by $T$. 
Assume w.l.o.g. that the root has $D \geq 1$ children denoted as $1, \ldots, D$. Denote by $T^{*}_u$ the subtree of all particles with common ancestor $u$.
We observe that if $\text{Br}\lr{T^{*}_u} < c$ for all $u$, then
$\text {Br}\lr{T} < c$.

Now, conditioned on the spin of the root, $\lr{T^{*}_u}_{u=1}^D$ are i.i.d. copies of $T^{\text {Poi}}$ with weight governed by the biased law $\NUstar$. The latter is a Galton-Watson process with offspring mean $\frac{a+(q-1)b}{q} \PHItwo$. If it dies out, then $\text{Br}\lr{T^{*}_u} = 0$ by definition. Hence, given that the process survives (and thus necessarily $\frac{a+(q-1)b}{q} \PHItwo > 1$), Proposition $6.4$ in \cite{Ly90} entails that $\text{Br}\lr{T^{*}_u} = \frac{a+(q-1)b}{q} \PHItwo$ a.s. 
\end{proof}

Note that it can in fact be easily proved that $\text{Br}\lr{T^{\text{Poi}}} = \frac{a+(q-1)b}{q} \PHItwo$ almost surely, given that the process survives.

We conclude with the main theorem of this section.  
\begin{theorem}
\label{thm::main_broadcasting}
Consider the multi-type branching process $T^{\emph{Poi}}$, where the root has spin drawn uniformly from $\spm$ and weight governed by $\nu$.  Denote the branching process by $T$ and its spins by $\tau^n$. Further, let $R$ be an unbounded non-decreasing function. Assume that $(a-b)^2 \PHItwo < q(a+b)$, then, for any $s \in \spm$,
\[  \P{\tau_{\rho} = s \left|T_{R(n)}, \tau_{\partial T_{R(n)}} \right.}  \overset{\mathbb{P}} \to \frac{1}{q}, \]
as $n \to \infty$.
\end{theorem}
\begin{proof}
Since $\epsilon = \frac{b}{a+(q-1)b}$, Proposition \ref{prop::branching} gives that Br$(\calT)\frac{(1 - q \epsilon)^{2}}{1-(q-2)\epsilon} < 1$ almost surely. Theorem \ref{thm::Evans_general} (and Remark \ref{rem::Evans_general}) then completes the proof. 
\end{proof}

\begin{remark}
In \eqref{eq::ComTree} we use a coupling between the Poisson tree and the local neighbourhood around a fixed vertex $u$, while we condition on the spins of all vertices \emph{exactly} distance $R(n)$ away from $u$. If there are no such vertices, i.e., when the neighbourhood 'dies out', then this does not entail extra information. Hence the convention that Br$(\calT)=0$ for a finite tree $\calT$.
\end{remark}

\section{Coupling of local neighbourhood }
\label{sec::coupling}
This section has as its objective to establish a coupling between the local neighbourhood of an arbitrary fixed vertex in the DC-SBM and $T^{\text{Poi}}$. The main result is  the following theorem, where we let $T$, $\tau$, and $\psi$ be random instances of $T^{\text{Poi}}$, its spins and its weights, respectively.

\begin{theorem}
Let $\rho$ be a uniformly picked vertex in $V(G)$, where for each $n$, $G = G(n)$ is an instance of the DC-SBM. There exists an unbounded non-decreasing function $R: \mathbb{N} \to \mathbb{N}$ such that 
\[\| \lr{G_{R(n)}(\rho), \sigma_{G_{R(n)}}, \phi_{G_{R(n)}}} -
  \lr{ T_{R(n)} , \tau_{T_{R(n)}}, \psi_{T_{R(n)}}} \|_{\emph{TV}}
  = 1 - o_n(1),\] 
and, 
\[ \P{|G_{R(n)}| \leq n^{1/9}} = 1 - o_n(1). \]
\label{thm::coupling}
\end{theorem}
\begin{remark}
In case the weights are bounded by some constant $\PHImax$, we can take $R(n) =  C \log(n)$, with $C < \frac{1 - \log(4/e)}{3 \log(2 \cdot \phi_{\text{max}}^2 \cdot(a \vee b))}$ and show that the coupling error is bounded by $n^{- \frac{1}{2} \log (4/e)}$. See the version of September 2016 of this work on Arxiv.
\end{remark}
We defer its proof to the end of this section. It uses an alternative description of the branching process in Section \ref{sec::broadcasting}.
\subsection{Alternative description of branching process}
For notational convenience, we restrict ourselves here to the case of two communities only. The proof for a general number of communities follows then analogously.
We obtain an alternative description of the graph by considering a particle $u$ with spin $\sigma_u$ and weight $\phi_u$ to be of type $x_u = \phi_u \sigma_u \in S = -W \cup W$. We denote the law of $x_u$ by $\mu$, i.e., for $A \subset S$, $\mu(A) = \int_A \frac{1}{2} \mathrm{d} \nu(|x|).$ Two distinct vertices $u$ and $v$ are then joined by an edge with probability $\frac{\kappa(x_u, x_v)}{n}$, where 
 $\kappa: S \times S \to \mathbb{R}$ is defined for $(x,y) \in S \times S$ by 
\be \kappa(x,y) = |xy|\lr{1_{\{xy > 0\}}a + 1_{\{xy < 0\}}b }. \label{eq::standard_kernel}\ee

Analogously, we obtain the following equivalent description of the branching process: We begin with a single particle $o$ of type $x_o$ governed by $\mu$, giving birth to Poi$(\lambda_{x_o}(S))$ children, where for $x \in S$, and $A \subset S$, 
\be \lambda_x(A) = \int_A \kappa(x,y) \mathrm d \mu(y). \label{def::lambda_x} \ee
conditioned on $x_o$ the children have i.i.d. types governed by $\MUstar_{x_o}$ \footnote{Note that if $y$ has law $\MUstar_x$, then for any $A \subset W$, $\P{\text{sign}(y) = \text{sign}(x), |y| \in A} =\frac{a}{a+b} \int_A z \frac{\mathrm{d} \nu(z)}{\PHI} = \P{\text{sign}(y) = \text{sign}(x)} \P{|y| \in A}$. Hence, we can identify sign$(y)$ with the particle's spin and $|y|$ with its \emph{independent} weight.}, where for $x \in S$, and $A \subset S$,
\be \MUstar_{x}(A) 
= \frac{\lambda_x(A)}{\lambda_x(S)} 
=  \int_A \lr{\frac{a}{a+b}\IND{xy > 0}  +  \frac{b}{a+b}\IND{xy < 0}} |y| \frac{ \mathrm d \nu(|y|) }{\PHI}. \ee
For generation $t \geq 1$, all particles give birth independently in the following way: A particle with type $x^*$ is replaced in the next generation by Poi$(\lambda_{x^*}(S))$  children, again with i.i.d. types governed by $\MUstar_{x^*}$. 

In case of a general number of communities, we let $\mu$ be the product measure of the uniform measure on $\spm$ with the measure $\nu$. I.e., for $s \in \spm$ and $A \subset [\PHImin, \infty)$, we have $\mu(\{s\} \times A) = \frac{1}{q} \cdot \nu(A).$

In \cite{BoBeJa07} it is shown that local neighbourhoods of the graph are described by the above branching process, if we \emph{ignore} the types. (To be precise: the equivalent description used in \cite{BoBeJa07} is that a particle of type $x$ gives birth to  Poi$(\lambda_{x}(A))$  children with type in $A$, for any $A \subset S$. Those numbers are independent for different sets $A$ and different particles.)

 The coupling-technique in \cite{BoBeJa07} uses a discretization of $\kappa$ as an intermediate step, thereby losing some information: types in the tree deviate slightly from their counterparts in the graph. We shall therefore use another coupling method, presented below, so that \textbf{the types in graph and branching process are exactly the same}. 
\subsection{Coupling}
We use the following exploration process:
At time $m=0$, choose a vertex $\rho$ uniformly in $V(G)$, where $G$ is an instance of the DC-SBM. Initially, it is the only active vertex:  $\mathcal{A}(0) = \{\rho\}$. All other vertices are neutral at start: $\mathcal{U}(0) = V(G) \setminus \{\rho\}$. No vertex has been explored yet: $\mathcal{E}(0) = \emptyset$. At each time $m \geq 0$ we arbitrarily pick an active vertex $u$ in $\mathcal{A}(m)$ that has shortest distance to $\rho$, and explore all its neighbours in $\mathcal{U}(m),$ the set of unexplored vertices. If $uv \in E(G)$ for $v \in \mathcal{U}(m)$, then we set $v$ active in step $m+1$, otherwise it remains neutral. At the end of step $m$, we designate $u$ to be explored. Thus,
\[ \calE(m+1) = \calE(m) \cup \{u\}, \]
\[ \calA(m+1) = \lr{ \calA(m) \setminus \{u\}  } \cup \lr{ \mathcal{N}(u) \cap  \calU(m) }, \]
and,
\[ \calU(m+1) = \calU(m) \setminus \mathcal{N}(u). \]
Our aim in this section is to show that the exploration process and the branching process are equal upto depth $R(n)$ (defined in Theorem \ref{thm::coupling}) with probability tending to one for large $n$. We do this in two steps:

Firstly, we establish that the types of the vertices in $\calU(m)$ are i.i.d. with law $\mu^{(m)}$ (defined in \eqref{eq::mu_m} below) such that 
\[ \left| \left| \mu^{(m)} - \mu \right| \right|_{\text{TV}} = \bigO \lr{ n^{-  \beta / 8} + mn^{-3/4}}. \]
This is the content of the following:
\begin{lemma}
The following holds conditioned that all the weights are smaller than $n^{\alpha}$, with $\alpha = 1/8$:
Let $1, \ldots, m$ be the vertices in $\calE(m)$, with types $X_1 = x_1, \ldots, X_m=x_m$. Then, the vertices in $\calU(m)$ have i.i.d. types with law
$ \mu^{(m)} = \mu^{(m)}_{x_1, \ldots, x_m}, $
 where
\be \mathrm{d}\mu^{(m)}(\cdot) = \frac{g(\cdot) \mathrm{d} \mu_{\alpha}(\cdot) }{\int_S g(z) \mathrm{d} \mu_{\alpha}(z) }, \label{eq::mu_m} \ee
with $\mu_{\alpha}$ denoting the measure of the types conditioned that all weights are bounded by $n^{\alpha}$, and where,
\be g(\cdot)  = \prod_{i=1}^m \lr{1 - \frac{\kappa(x_i, \cdot)}{n}}. \label{eq::g} \ee
Further, for all $(x_1, \ldots, x_m)$:
\[ \left| \left| \mu^{(m)}_{x_1, \ldots, x_m} - \mu \right| \right|_{\text{TV}} = \bigO \lr{ n^{- \alpha \beta} + mn^{2\alpha - 1}}. \] \label{lm::IID_Reservoir}
\end{lemma}
Secondly, if $u$ has type $X = x \in S$, then its $D$ neighbours in $\calU(m)$ (i.e., those vertices that will be added to $\calA(m+1)$) have i.i.d. types with a law $\mu^{*(m+1)}_x$ (defined in \eqref{eq::nu_x_m} below), which is $\bigO \lr{n^{-3/8}}$ away from $\MUstar_x$ in total variation distance. Further, the total variation distance between the number of neighbours $D$ and 
$\text{Poi}\lr{  \lambda_x(S) }$ is $\bigO \lr{n^{-1/4}}$:

\begin{lemma}
The following holds conditioned that all the weights are smaller than $n^{\alpha}$, with $\alpha = 1/8$:
Assume $u$ has type $X=x$. Let $D$ be the number of neighbours $u$ has in $\calU(m)$. Then, the types of those neighbours are i.i.d. with law $\mu^{*(m)}_x$, where
\be \mathrm{d} \mu^{*(m)}_x(\cdot) =  \frac{ \ \kappa(x,\cdot) \mathrm{d} \mu^{(m)}(\cdot) }{  \int_S   \kappa(x,y) \mathrm{d} \mu^{(m)}(y) }  . \label{eq::nu_x_m}\ee
 For large $n$ and $m = o(n^{1/4})$,
\be \left| \left| \mu^{*(m)}_x - \MUstar_x \right| \right|_{\text{TV}} = \bigO \lr{ n^{\alpha(1 -  \beta)} + m n^{3 \alpha - 1} + n^{-\alpha \beta / 2}} = \bigO \lr{n^{-3/8}}. \label{eq::TV_nu_x_m} \ee
Further, 
\be \left| \left| D - \emph{Poi}\lr{  \lambda_x(S)} \right| \right|_{\text{TV}} = \bigO \lr{n^{(1 - \beta/2) \cdot 1/8} + n^{-1/4}} = \bigO \lr{n^{-1/4}}. \label{eq::TV_D} \ee
\label{lm::ActiveNeighbours}
\end{lemma}

To establish the desired coupling, we need to show that certain events happen with high probability. To define those events, we need some notation:
For $u \in \partial G_r$ (we identify $\partial G_r = \{ 1, \ldots, |\partial G_r| \}$), put
\[
D_u = |\mathcal{N}(u) \cap \mathcal{U}(|G_{r-1}| + u -1)|. \]
Conditioned that $u$ has type $X_u = x_u$, let
\[\widehat{D}_u = \text{Poi}\lr{  \lambda_{x_u}(S) }.\]
Further, for $v \in \{1, \ldots, D_u \}$, let $U_{uv}$ denote the type of child $v$ of vertex $u$ and let $\widehat{U}_{uv}$ be a random variable with law $\MUstar_{x_u}$. We assume that $\{ \widehat{U}_{uv} \}_v$ are independent conditioned on $X_u = x_u$.

We put the function $g: s \mapsto 2^s - 1$ and define the events
\[ A_{r+1} = \{ \forall u \in \partial G_r : D_u = \widehat{D}_u \}, \]
\[ B_{r+1} = \{ \forall u \in \partial G_r, v \in \{1, \ldots, D_u \} :  U_{uv} = \widehat{U}_{uv} \}, \]
\[ C_{r} = \{ |\partial G_s| \leq \log^{g(s)}(n) \  \forall s \leq r \}, \]
and their intersection \[E_{r} =  \bigcap_{s = 1}^r \{ A_{s} \cap B_{s} \cap C_{s} \}. \]
Further, we let $K_r$ be the event that no vertex outside $G_{r}$ has more than one neighbour in $G_{r}$ and that there are no edges in $\partial G_r$ (this implies that the neighbourhood is indeed a tree).

The events $E_{r}$ and $K_r$ happen with high probability:
\begin{lemma}
The following holds conditioned that all the weights are smaller than $n^{\alpha}$, with $\alpha = 1/8$: \emph{Fix} $R \geq 0$. Then, for $r  \leq R$, 
\[ \mathbb{P} \lr{E_{r+1}| E_{r}} = 1 - o_n(1). \]
\label{lm::InductionE_R}
\end{lemma}
\begin{lemma} 
The following holds conditioned that all the weights are smaller than $n^{\alpha}$, with $\alpha = 1/8$: \emph{Fix} $R \geq 0$. Then, for $r  \leq R$,
\[ \P{K_r | C_{R}} = 1 - o_n(1). \]
\label{lm::D_r}
\end{lemma}
\begin{proof}[Proof of Lemma \ref{lm::IID_Reservoir}]
Recall that we assume that all weights are bounded by $n^{\alpha}$. Consider vertex $v \in \mathcal{U}(m)$ with type $Y$. 
We show first that, conditioned on $v \notin \mathcal{N}(1, \ldots, m)$ and $X_1 = x_1, \ldots, X_m = x_m$, $Y$ has law $\mu^{(m)}_{x_1, \ldots, x_m}$.
From Bayes theorem we have, for $y \in S$, 
\be \ba  & \P{Y \leq y | v \notin \mathcal{N}(1, \ldots, m) ,  X_1 = x_1, \ldots, X_m = x_m} \\  &\quad =  \frac{\P{Y \leq y } \P{v \notin \mathcal{N}(1, \ldots, m) | Y \leq y, X_1 = x_1, \ldots, X_m = x_m}}{\P{v \notin \mathcal{N}(1, \ldots, m) | X_1 = x_1, \ldots, X_m = x_m}},
\ea \label{eq::int_mu_m}\ee
since $\P{Y \leq y | X_1 = x_1, \ldots, X_m = x_m} = \P{Y \leq y }$.
Recall \eqref{eq::g} and observe that 
\[ g(\cdot) = \P{v \notin \mathcal{N}(1, \ldots, m) | Y = \cdot, X_1 = x_1, \ldots, X_m = x_m}.  \]
Hence, the denominator in \eqref{eq::int_mu_m} is just $\int_S g(z) \mathrm{d} \mu(z)$ and 
evaluating the numerator yields $\int_{-\infty}^y g(z) \mathrm{d} \mu(z)$. We thus obtain \eqref{eq::mu_m}.

Since for $|y| \leq \bigO \lr{n^{\alpha}}$,  $\mathrm{d}\mu_{\alpha}(y) = \frac{\mathrm{d}\mu(y)}{\P{\phi \leq n^{\alpha}}} $, it follows that $\| \mu_{\alpha} -  \mu \|_{\text{TV}} = \bigO \lr{n^{-\alpha \beta}}$.

To bound $\| \mu_{\alpha} -  \mu^{(m)} \|_{\text{TV}}$, note that (in view of \eqref{eq::standard_kernel}) $g(y) = 1 - \bigO \lr{m n^{2\alpha - 1}}$, for $|y| \leq \bigO \lr{n^{\alpha}}$. Thus, $I:= \int_S g(z) \mathrm{d} \mu_{\alpha}(z) = 1 - \bigO \lr{m n^{2\alpha - 1}}.$ Therefore, 
\[ \ba
\left|\left|\mu^{(m)} - \mu_{\alpha} \right|\right|_{\text{TV}} &\leq \int_S \left| \frac{g(y)  }{ I } - 1\right| \mathrm{d} \mu_{\alpha}(y)    
= \bigO \lr{m n^{2\alpha - 1}}.
 \ea \] 
We finish by invoking the triangle inequality. 
\end{proof}

\begin{proof}[Proof of Lemma \ref{lm::ActiveNeighbours}]
Put $n_m = |\mathcal{U}(m)|$ and let  $ Y_1, \ldots, Y_D$ denote the types of the neighbours of $u$. 

 Let $f_1, \ldots, f_n$ be arbitrary measurable functions. The first claim follows if we prove that 
\be \E{ \left. \mathrm{e}^{ - \s{j=1}{D} f_j(Y_j) } \right| D = d } = \prod_{j=1}^d \lr{ \int_S  \mathrm{e}^{-f_j(y)}  \mathrm{d} \mu^{*(m)}_x(y) }. \label{eq::iid} \ee
Now, abbreviating conditioning on $\mathcal{N}(u) \cap \mathcal{U}(m) = F$ by $F$, we have,
\[ \ba
& \E{  \mathrm{e}^{ - \s{j=1}{D} f_j(Y_j) } 1_{D=d} } \\
&= \sum_{F \subset [n_m], |F| = d}  \E{ \left.  \mathrm{e}^{ - \s{j \in F}{} f_j(Y_j) } \right| F } \cdot \lr{1 - \frac{1}{n} \int_S \kappa(x,y) \mathrm{d} \mu^{(m)}(y)}^{n_m - d} \\ &\quad \cdot \lr{\frac{1}{n} \int_S \kappa(x,y) \mathrm{d} \mu^{(m)}(y)}^{ d}. 
 \ea \]
We have, 
\[ \ba \P{D = d} &=   {n_m \choose d} \lr{ 1 - \frac{1}{n} \int_S  \kappa(x,y) \mathrm{d} \mu^{(m)}(y) }^{n_m - d} \\ &\quad\quad \cdot \lr{\frac{1}{n} \int_S  \kappa(x,y) \mathrm{d} \mu^{(m)}(y) }^{d}. \ea \]
Hence, 
\[ \E{ \left. \mathrm{e}^{ - \s{j=1}{D} f_j(Y_j) } \right| D = d } = \frac{1}{{n_m \choose d}} \sum_{F \subset [n_m], |F| = d} \E{ \left.  \mathrm{e}^{ - \s{j \in F}{} f_j(Y_j) } \right| F }.\]
Conditioned on  $F \subset [n_m]$, the types $(Y_j)_{j \in F}$ are i.i.d., thus 
\[ \ba 
\E{ \left.  \mathrm{e}^{ - \s{j \in F}{} f_j(Y_j) } \right| F } = \prod_{j=1}^d \lr{ \frac{ \int_S  \mathrm{e}^{-f_j(y)} \frac{\kappa(x,y)}{n} \mathrm{d} \mu^{(m)}(y) }{  \int_S   \frac{\kappa(x,y)}{n} \mathrm{d} \mu^{(m)}(y) } },
\ea \]
which combined with \eqref{eq::nu_x_m} gives \eqref{eq::iid}, our first claim.

 Further, 
 \BA \| \mu^{*(m)}_x - \MUstar_x \|_{\text{TV}} &\leq   \int_S f_x(y) \left|  \frac{\mathrm{d} \mu^{(m)} (y)}{I_x^{(m)}} - \frac{\mathrm{d} \mu(y)}{I_x} \right| \\
 &= \frac{1}{I_x} \int_S f_x(y) \left|  \mathrm{d} \mu^{(m)} (y) \lr{1 + \bigO \lr{I_x^{(m)} - I_x }} - \mathrm{d} \mu(y) \right|  
, \label{eq::TV_biased}\EA
where $f_x(y) = \lr{ \indicator{xy > 0}a + \indicator{xy < 0}b} |y|, $
$I_x^{(m)} = \int_S f_x(z) \mathrm{d} \mu^{(m)} (z)$ and
$I_x = \int_S f_x(z) \mathrm{d} \mu (z)$. 
Now, 
\BA |I_x^{(m)} - I_x| &\leq \bigO \lr{n^{\alpha}} \int_{|z| \leq n^{\alpha}} |\mathrm{d} \mu^{(m)} (z) - \mathrm{d} \mu (z)| +  \int_{|z| > n^{\alpha}} |z| \mathrm{d} \mu (z) \\
&= \bigO \lr{ n^{\alpha - \alpha \beta} + m n^{3 \alpha - 1} + n^{-\alpha \beta / 2}},  \EA 
where we used the proof of the previous lemma to bound the first term and Cauchy-Schwartz inequality for the second term. Now, the right-hand side in \eqref{eq::TV_biased} is thus of the same order (since the weights have expectation).  

For the last claim, observe that $D = $ Bin$(n_m,p)$, where \\$p  = \frac{1}{n} \int_S \kappa(x,y) \mathrm{d} \mu^{(m)}(y)$. Hence, since the weights have bounded first moment, 
\[ \left| \left| \text{Bin}(n_m,p) - \text{Poi}\lr{ n_m p } \right| \right|_{\text{TV}} \leq \s{i=1}{n_m} p^2 = \bigO \lr{n^{-3/4}}. \]
Standard bounds for Poisson random variables entail the existence of a constant $C_{\text{Poi}} \geq 1$ such that $||$Poi$(\mu) - $ Poi$(\lambda)||_{\text{TV}} \leq C_{\text{Poi}}|\mu - \lambda|$. Consequently,
\[ \ba \frac{1}{C_{\text{Poi}}} \left| \left| \text{Poi}(n_m p) - \text{Poi}\lr{ \lambda_x(S) } \right| \right|_{\text{TV}} &\leq |n_m-n|p + |x| |I_x^{(m)} - I_x| \\
&\leq \frac{|n_m - n|}{n} n^{\alpha}  \\
&\quad + \bigO\lr{ n^{2\alpha - \alpha \beta} + m n^{4 \alpha - 1}  + n^{\alpha - \alpha \beta/2} }. \ea \]
Thus, by the triangle inequality,
\[ \left| \left| \text{Bin}(n_m,p) - \text{Poi}\lr{  \lambda_x(S) } \right| \right|_{\text{TV}} = \bigO \lr{n^{(1 - \beta/2) \cdot 1/8} + n^{-1/4}}.  \]

\end{proof}

\begin{proof}[Proof of Lemma \ref{lm::InductionE_R}]
Write $n_r = |\partial G_r|$. 
We have
\[ \P{E_{r+1}| E_{r}} \geq \P{  B_{r+1}| E_{r}} - \P{\neg A_{r+1}| E_{r}} - \P{ \neg C_{r+1}| E_{r}} . \]
Now, 
\be \ba
\P{  B_{r+1}| E_{r},n_r} 
&\geq 1 -  \sum_{u=1}^{n_r} \P{ \left.  \neg B^{(u)}_{r+1} \right| \bigcap_{v=1}^{u-1} B^{(v)}_{r+1} ,E_{r}},
 \ea \label{eq::Br}\ee
where $ B^{(u)}_{r+1} = \{ \forall w \in \{1, \ldots, D_u \} :  U_{uw} = \widehat{U}_{uw} \}. $
Denote the already explored vertices by $1, \ldots, m$ (where $m = |G_{r-1}| + u - 1$) and their types as $X_1, \ldots, X_m$. 
Conditioned on those types, the vertices in $\mathcal{U}(m)$ are i.i.d. with distribution $\mu^{(m)}$. Hence:
\be \ba
  &\P{ \left.   B^{(u)}_{r+1} \right| \bigcap_{v=1}^{u-1} B^{(v)}_{r+1} ,E_{r}, n_r, X_1, \ldots, X_m}  
 =  \P{ \left.   B^{(u)}_{r+1} \right|  X_1, \ldots, X_m } \\
 &\geq \P{\left.    B^{(u)}_{r+1} \right| D_u \leq \log(n) \log^{g(r)}(n),  X_1, \ldots, X_m  } \\
 &\quad\quad \cdot  \P{ \left. D_u \leq \log(n) \log^{g(r)}(n) \right|  X_1, \ldots, X_m  }.
 \ea \label{eq::Br_j} \ee
 Now, $D_u \overset{d} \leq \sum_{i=1}^n \text{Ber}\lr{(a+b)\frac{\phi^* \phi_i}{n}}$, where $\phi^*$ is governed by the size-biased law $\NUstar$ and $\{\phi_i\}_i$ are i.i.d. and bounded by $n^{\alpha}$. Hoeffding's inequality gives that $\frac{1}{n} \sum_{i=1}^n \phi_i \leq 2 \PHI$ w.p. at least $1 - \exp(-n^{1-2 \alpha})$, and $\phi^* \leq \log^{g(r)}(n)$ w.p. at least $1 - \bigO \lr{ \lr{\log^{g(r)}(n)}^{1 - \beta} }$ (note the exponent  $\beta - 1$ of the size-biased power-law). Conditioned on those events, we use a multiplicative Chernoff bound to obtain,
 \be \ba 
 & \P{\left. D_u \leq \log(n) \log^{g(r)}(n) \right|  X_1, \ldots, X_m  } 
 &\geq 1 - \bigO \lr{ \lr{\log^{g(r)}(n)}^{1 - \beta} }.
 \ea \label{eq::bin} \ee 
   Lemma \ref{lm::ActiveNeighbours} entails, since $m = o(n^{1/4})$,
 \be  \ba  \P{    B^{(u)}_{r+1} \left| D_u \leq \log^{g(r)+1}(n), \  X_1, \ldots, X_m \right. } 
 \geq 1 - \bigO \lr{ \frac{ \log^{g(r)+1}(n) } {n^{3/8}} } \ea. \label{eq::B}\ee
Then, \eqref{eq::Br_j} - \eqref{eq::B} together give
 \[ \P{    B^{(u)}_{r+1} \left| \bigcap_{v=1}^{u-1} B^{(v)}_{r+1} ,E_{r}, X_1, \ldots, X_m \right.} \geq 1 - \bigO \lr{ \lr{\log^{g(r)}(n)}^{1 - \beta} }. \]
Now, since conditioned on $E_r$, $n_r \leq \log^{g(r)}(n)$, \eqref{eq::Br} gives
 \[ \ba \P{    B_{r+1} | E_{r}} &\geq 1 - \bigO \lr{ \lr{\log^{g(r)}(n)}^{2 - \beta} }.  \ea \] 
The growth condition ($C_r$)follows also from \eqref{eq::bin}.

We take a similar approach to quantify  
 \be \ba
\P{  A_{r+1}| E_{r},n_r} \geq 
 1 - \sum_{u=1}^{n_r} \P{ \left.  \neg A^{(u)}_{r+1} \right| \bigcap_{v=1}^{u-1} A^{(v)}_{r+1} ,E_{r},n_r},
 \ea \label{eq::Ar} \ee
 where, 
 $ A^{(u)}_{r+1} = \{ D_u = \widehat{D}_u, D_u \leq \log^{g(r)+1}(n) \}. $
 Now, 
 \be \ba &  \P{ \left.   A^{(u)}_{r+1} \right| \bigcap_{v=1}^{u-1} A^{(v)}_{r+1} ,E_{r}} 
&\geq 1 - \bigO \lr{n^{(1-\beta/2) 1/8} + n^{-1/4} +   \log^{g(r)(1-\beta)}(n)   },
 \ea  \ee
due to Lemma \ref{lm::ActiveNeighbours}, since $n - |\mathcal{U} (m)| = o(n^{1/4})$ when $r$ is fixed. 
Thus, \eqref{eq::Ar} gives
\[ \P{  A_{r+1}| E_{r}} \geq 1- \bigO \lr{\log^{g(r)}(n) n^{(1-\beta/2) 1/8} + n^{-1/4} +   \log^{g(r)(2-\beta)}(n)   }. \]
\end{proof}

\begin{proof}[Proof of Lemma \ref{lm::D_r}]
Fix $u,v \in \partial G_r$. The probability of having an edge between $u$ and $v$ is smaller than $\bigO \lr{  n^{2 \alpha - 1} }$. For any $w \in V(G \setminus G_r)$, the probability that $(u,w)$ and $(v,w)$ both appear is smaller than $\bigO \lr{  n^{4 \alpha - 2} }$. Now, Lemma \ref{lm::InductionE_R} implies that
\[ |G_r| \leq \log(n)^{g(R)} R = \log^{2^R - 1}(n) R. \] 
 Hence, the result follows from a union bound over all triples $u,v,w$.
\end{proof}

\begin{proof}[Proof of Theorem \ref{thm::coupling}]
We can assume that all weights are bounded by $n^{\alpha}$. Indeed, by a  union bound over all vertices, this happens with probability $1 - \bigO \lr{ n^{1 - \alpha \beta}} = 1 - o_n(1).$ 
For a fixed integer $R > 0$, we have 
\[ \P{\cap_{s=1}^R K_s ,E_R} = 1 - o_n(1). \]
We construct a sequence $\{ N_k \}_{k=0}^{\infty}$ inductively as follows: Put $N_0=0$ and for each $k$, $N_k > N_{k-1}$ as the smallest number such that
\[ \P{ \cap_{s=1}^k K_s, E_k } \geq 1 - \frac{1}{k}, \text{ and } \log^{2^k - 1}(n) k \leq n^{1/9},\]
for all $n \geq N_k$. Put for $N_k \leq n < N_{k+1}$, $R(n) = k$. Then, for $n \geq N_k$,
\[ \P{ \cap_{s=1}^{R(n)} K_s, E_{R(n)}, |G_{R(n)}| \leq n^{1/9} } \geq 1 - \frac{1}{k}.\]
\end{proof}
\section{No long-range correlation in DC-SBM}
\label{sec::weak}
In this section we establish the main Theorem \ref{thm::ConvergenceToHalf}, from which Theorem \ref{thm::Main} then follows. To this end, we first condition on both the spins  of $\partial G_{R(n)}$ and all weights in $G$. Lemma \ref{lm::MarkovField} below shows that we then can remove the conditioning on $\sigma_v$ and the graph structure outside the $R$-neighbourhood (including the weights):
\be \P{\sigma_u =+| \sigma_{\partial G_{R}}, \sigma_v , G, \phi}  = \P{\sigma_u =+| \sigma_{\partial G_{R}} , G_R, \phi_{G_R}} + o_n(1). \label{eq::CondOnBound}\ee
We established in the previous section that a neighbourhood in $G$ looks like a $T^{\text{Poi}}$ tree with a Markov broadcasting process on it. Hence, the right-hand side of \eqref{eq::CondOnBound} converges to $	1/q$ in probability, establishing \eqref{eq::con}. We show in Lemma \ref{lm::NoConvergence} below that this contradicts the existence of a reconstruction that is positively correlated with the true type-assignment.

We begin by preparing an auxiliary lemma to prove \eqref{eq::con}, it establishes that long-range interactions are sufficiently weak. Its proof is inspired by Lemma 4.7 in \cite{MoNeSly15}. However (besides the additional complication of weights) the result stated here is stronger in the sense that the $o_n(1)$ terms converge uniformly to $0$ and that "conditioning on $G$" may now be replaced with "conditioning on $G_{A \cup B}$".
\begin{lemma}
The following holds conditioned that all the weights are smaller than $n^{\alpha}$, with $\alpha = 1/8$: Let  $G$ be an instance of the DC-SBM. Let  $s \in \spm$. Let $u$ be an uniformly picked vertex in $V(G)$.
Let $A = A(G)$, $B = B(G)$, $C = C(G) \subset V$ be a (random) partition of $V(G)$, with $u \in A$, such that $B$ separates $A$ and $C$ in $G$. Assume that $|A \cup B| \leq  n^{1/9}$ for asymptotically almost every realization of $G$. Then there exists a sequence of events $(\Omega_n)_n$ and a sequence of non-negative real numbers $(\epsilon_n)_n$, such that $\P{\Omega_n} = 1 - o_n(1)$, and $\epsilon(n)=o_n(1)$, and further, for each $n$,
\be |  \P{\sigma_u =s | \sigma_{B \cup C} , G,\phi} - \P{\sigma_u =s | \sigma_B ,  G_{A \cup B}, \phi_{A \cup B}}| \leq \epsilon(n), \label{eq::weak_long_range} \ee 
on $\Omega_n$.
\label{lm::MarkovField}
\end{lemma}
\begin{proof}For a fixed graph $g$, spin-configuration $\tau$ and degree-configuration $\psi$, we make a factorization of $\P{G=g,\sigma = \tau | \phi = \psi}$ into parts depending on $A,B$ and $C$. 
We claim that the part that measures the interaction between $A$ and $C$ is asymptotically independent of $\tau$. 
Put
\[ \Psi_{uv}(g,\tau, \psi) = \left\{ 
  \begin{array}{l l}
    a \frac{\psi_u \psi_v}{n} & \quad \text{ if } (u,v) \in E(g) \text{ and } \tau_u = \tau_v \\
    b \frac{\psi_u \psi_v}{n} & \quad \text{ if } (u,v) \in E(g) \text{ and } \tau_u \neq \tau_v  \\
    1- a \frac{\psi_u \psi_v}{n} & \quad \text{ if } (u,v) \notin E(g) \text{ and } \tau_u = \tau_v \\
    1 - b \frac{\psi_u \psi_v}{n} & \quad \text{ if } (u,v) \notin E(g) \text{ and } \tau_u \neq \tau_v.  \\
  \end{array} \right. \]
We define for arbitrary sets $U_1, U_2 \subset V$, 
\[ \ba Q_{U_1, U_2} &= Q_{U_1, U_2}(g,\tau,\psi) = Q_{U_1, U_2}(g_{U_1 \cup U_2},\tau_{U_1 \cup U_2},\psi_{U_1 \cup U_2}) \\
&= \prod_{u \in U_1, v \in U_2} \Psi_{uv}(g,\tau, \psi), \ea \]
where the subscript indicates restriction of the corresponding quantities to $U_1 \cup U_2$.
Then, we have,
\be  \P{G=g |\sigma=\tau, \phi = \psi} = Q_{A \cup B, A \cup B} Q_{B \cup C, C} Q_{A,C}. \label{eq::prob} \ee

We begin by demonstrating that $Q_{A,C}$ is asymptotically independent of $\tau$: Write, 
\[ Q_{A,C}(g,\tau,\psi) = \prod_{u \in A, v \in C: \tau_u = \tau_v} \lr{ 1- a \frac{\psi_u \psi_v}{n}} \prod_{u \in A, v \in C: \tau_u \neq \tau_v} \lr{ 1- b \frac{\psi_u \psi_v}{n}}, \] since $A$ and $C$ are separated by $B$ (there are thus no edges between $A$ and $C$).
The first product may be rewritten as, 
\[ \ba \prod_{u \in A, v \in C: \tau_u = \tau_v} \lr{ 1- a \frac{\psi_u \psi_v}{n}} &= \text{exp}\lr{\s{u \in A, v \in C: \tau_u = \tau_v}{} \text{log}\lr{1- a \frac{\psi_u \psi_v}{n}}} \\
&= \text{exp}\lr{\s{u \in A, v \in C: \tau_u = \tau_v}{} \lr{ - a \frac{\psi_u \psi_v}{n} + \bigO \lr{ n^{4\alpha - 2} }}} \\
&= \text{exp}\lr{ - \frac{a }{n}\s{u \in A, v \in C: \tau_u = \tau_v}{} \psi_u \psi_v +  \bigO \lr{ n_A  n^{4\alpha - 1}} } .
\ea \]
Now, the sum $\frac{1}{n}\s{u \in A, v \in C: \tau_u = \tau_v}{} \psi_u \psi_v$ tends to $\frac{\|A\| \PHI}{q}$, if $(\tau,\psi) \in \Omega(n)$, where
\[ \|A\| = \s{u \in A}{} \psi_u, \]
and where,
\be \Omega(n) = \left \{ (\tau',\psi'):  \left| \frac{1}{n}\sum_{\tau_u = k, u \in V} \psi_u - \frac{\PHI}{q} \right| \leq n^{- \frac{1}{4}}, \forall k \in \spm \right \}. \label{eq::omega} \ee
Indeed, 
\BA \frac{1}{n} \s{u \in A, v \in C: \tau_u = \tau_v}{} \psi_u \psi_v &= \sum_{k=1}^q \s{u \in A}{} \indicator{\tau_u = k} \psi_u \ \frac{1}{n} \s{v \in C}{} \indicator{\tau_v = k} \psi_v  \\
&= \frac{\|A\| \PHI}{q} + \bigO\lr{ n^{- \frac{1}{72}}}, \EA
since $|V|-|C| \leq n^{1/9}$ and $\psi_u \leq n^{1/8}.$

As a consequence,
\[ \ba \prod_{u \in A, v \in C: \tau_u = \tau_v} \lr{ 1- a \frac{\psi_u \psi_v}{n}} &=  \exp{ \lr{ \bigO \lr{n^{-\frac{1}{72}}} }} \cdot \exp \lr{ -  a \frac{ \|A\| \PHI}{q}  } \\ 
&= (1 + o_n(1))\exp{ \lr{  - a \frac{\|A\| \PHI}{q} } },  \ea \] 
where the $o_n$ term is uniform for all $(\tau,\psi) \in \Omega(n)$. We carry out a similar calculation for the other product. Together we obtain
\be Q_{A,C}(g,\tau,\psi) = (1 + o_n(1))\exp{ \lr{ -   \frac{a+(q-1)b}{q}  \|A\| \PHI  } }, \label{eq::Q_AC}\ee
uniformly for all $(\tau,\psi) \in \Omega(n)$. This proves that $Q_{A,C}(g,\tau,\psi)$ is indeed essentially independent of $\tau$ for most pairs $(\tau, \psi)$.

We use the above to prove that, for $u \in V$, 
\be \ba &\P{\sigma_u=\tau_u | \sigma_{B \cup C} = \tau_{B \cup C}, G=g, \phi = \psi, ( \phi, \sigma) \in \Omega(n)} \\ &= (1 + o_n(1)) \P{\sigma_u=\tau_u | \sigma_B = \tau_B, G_{A \cup B}=g_{A \cup B}, \phi_{A \cup B} = \psi_{A \cup B}, ( \phi, \sigma) \in \Omega(n)} \\ &\quad +o_n(1). \ea \label{eq::CondOmega} \ee 
Fix  $(\tau,\psi) \in \Omega(n)$. Then, 
\be \P{G=g,\sigma=\tau | \phi = \psi, ( \phi, \sigma) \in \Omega(n) } =\P{G=g | \sigma=\tau, \phi = \psi} f(\psi,n), \label{eq::probInd}
 \ee
where $f(\psi,n) = \P{\sigma = \tau |\phi = \psi, ( \phi, \sigma) \in \Omega(n)} = \frac{q^{-n}}{\P{( \phi, \sigma) \in \Omega(n) | \phi = \psi}}$. 
 Hence, plugging \eqref{eq::prob} and \eqref{eq::Q_AC} in \eqref{eq::probInd},
\be \ba &\P{G=g,\sigma=\tau | \phi = \psi, ( \phi, \sigma) \in \Omega(n)} \\ \quad &=  Q_{A \cup B, A \cup B} (g,\tau,\psi) Q_{B \cup C, C} (g,\tau,\psi) \\
&\quad\quad \cdot (1 + o_n(1))\exp{ \lr{ -   \frac{a+(q-1)b}{q} \|A\| \PHI } } f(\psi,n). \label{eq::plug} \ea \ee
 Put, for $U \subset V$,
\[ \Omega_U(n) = \Omega_U( \psi, \tau_U, n) = \{ \tau': \tau'_U = \tau_U, (\tau',\psi) \in \Omega(n)  \}, \] 
then, invoking \eqref{eq::plug},
\be \ba
& \P{G=g,\sigma_U=\tau_U | \phi = \psi, ( \phi, \sigma) \in \Omega(n)} \\ &= \sum_{\tau' \in \Omega_U(n)} \P{G=g,\sigma=\tau' | \phi = \psi, ( \phi, \sigma) \in \Omega(n)} \\
&= \sum_{\tau' \in \Omega_U(n)}
Q_{A \cup B, A \cup B} (g,\tau',\psi) Q_{B \cup C, C} (g,\tau',\psi) \\
&\quad\quad \cdot (1 + o_n(1))  \exp{ \lr{ -   \frac{a+(q-1)b}{q} \|A\| \PHI } } f(\psi,n) \\
&= (1 + o_n(1))  \exp{ \lr{ -   \frac{a+(q-1)b}{q} \|A\| \PHI } } f(\psi,n) \\
&\quad\quad \cdot \sum_{\tau' \in \Omega_U(n)}
Q_{A \cup B, A \cup B} (g,\tau',\psi) Q_{B \cup C, C} (g,\tau',\psi), \\
 \ea \label{eq::rel} \ee
where we could interchange the order $o_n(1)$ term and the sum because the former holds \emph{uniformly} for all $(\phi,\sigma) \in \Omega(n)$. 

 We apply \eqref{eq::rel} with $U=A$ and $U = A \cup B$, to rewrite the right hand side of
\be \ba 
&\P{\sigma_A=\tau_A | \sigma_B = \tau_B, G=g, \phi = \psi, ( \phi, \sigma) \in \Omega(n)} \\
&= \frac{\P{G=g, \sigma_{A \cup B}=\tau_{A \cup B} |  \phi = \psi, ( \phi, \sigma) \in \Omega(n)}}{\P{ G=g,  \sigma_{B}=\tau_{B}|  \phi = \psi, ( \phi, \sigma) \in \Omega(n)}} \label{eq::prob_condB}
\ea \ee
as
\[ \ba & \quad (1 + o_n(1)) \frac{\sum_{\tau' \in \Omega_{A \cup B}(n)}
Q_{A \cup B, A \cup B} (g,\tau',\psi) Q_{B \cup C, C} (g,\tau',\psi)}{\sum_{\tau' \in \Omega_{B}(n)}
Q_{A \cup B, A \cup B} (g,\tau',\psi) Q_{B \cup C, C} (g,\tau',\psi)} \\
&= (1 + o_n(1)) \frac{Q_{A \cup B, A \cup B} (g,\tau,\psi) \sum_{\tau' \in \Omega_{A \cup B}(n)}
 Q_{B \cup C, C} (g,\tau',\psi)}{\sum_{\tau''' \in \Omega_{B \cup C}(n)}
Q_{A \cup B, A \cup B} (g,\tau''',\psi) \sum_{\tau'' \in \Omega_{A \cup B}(n)} Q_{B \cup C, C} (g,\tau'',\psi)}, \ea \]
where we used that $Q_{U_1,U_2}(\tau')$ depends  on $\tau'$ only through $\tau'_{U_1 \cup U_2}$ to rewrite the numerator. Factorization of the denominator is justified as follows:
For an arbitrary $\tau' \in \Omega_B(n)$, put $\tau'' = (\tau_{A \cup B},\tau'_C) \in  \Omega_{A \cup B}(n)$ and $\tau''' = (\tau'_A, \tau_{B \cup C}) \in  \Omega_{B \cup C}(n)$. Then,
\be Q_{A \cup B, A \cup B} (g,\tau',\psi) Q_{B \cup C, C} (g,\tau',\psi) = Q_{A \cup B, A \cup B} (g,\tau''',\psi) Q_{B \cup C, C} (g,\tau'',\psi). \label{eq::factor_denom}\ee
This proves that the double summation is at least as large as the single sum. Equality follows upon putting $\tau' = (\tau'''_A,\tau_B,\tau''_C)$ for arbitrary $\tau'' \in \Omega_{A \cup B}(n)$ and $\tau''' \in \Omega_{B \cup C}(n)$: \eqref{eq::factor_denom} is then again satisfied. Hence, \eqref{eq::prob_condB} is equivalent to
\be \ba  &\P{\sigma_A=\tau_A | \sigma_B = \tau_B, G=g, \phi = \psi, ( \phi, \sigma) \in \Omega(n)} \\
&= (1 + o_n(1)) \frac{Q_{A \cup B, A \cup B} (g,\tau,\psi)}{\sum_{\tau''' \in \Omega_{B \cup C}(n)}
Q_{A \cup B, A \cup B} (g,\tau''',\psi) }. \label{eq::prob_condB_2}  \ea \ee

We shall rewrite the right hand side of \eqref{eq::prob_condB_2} to obtain on the one hand:
\be \ba  &\P{\sigma_u=\tau_u | \sigma_B = \tau_B, G=g, \phi = \psi, ( \phi, \sigma) \in \Omega(n)} \\
&= (1 + o_n(1))  \widehat{F}\lr{g_{A \cup B},\tau_{u \cup B},\psi_{A \cup B}}, \label{eq::Fhat} \ea \ee
for some function $\widehat{F}(\cdot) \leq 1$.
And, on the other hand:
\be \ba & \P{\sigma_u=\tau_u | \sigma_B = \tau_B, G=g, \phi = \psi, ( \phi, \sigma) \in \Omega(n)} \\ &= (1 + o_n(1)) \P{\sigma_u=\tau_u | \sigma_{B \cup C} = \tau_{B \cup C}, G=g, \phi = \psi, ( \phi, \sigma) \in \Omega(n)}. \ea \label{eq::sigmaBC} \ee
To do so, note that
\[ \sum_{\tau''' \in \Omega_{B \cup C}(n)} Q_{A \cup B, A \cup B} (g,\tau''',\psi) =
\sum_{\tau_A''' \in \spm^A } Q_{A \cup B, A \cup B} (g_{A \cup B},(\tau_A''',\tau_B),\psi_{A \cup B}),
 \]
Therefore, \eqref{eq::prob_condB_2} is equivalent to
\[\ba  \P{\sigma_A=\tau_A | \sigma_B = \tau_B, G=g, \phi = \psi, ( \phi, \sigma) \in \Omega(n)} = (1 + o_n(1))  F\lr{g_{A \cup B},\tau_{A \cup B},\psi_{A \cup B}}, \ea \]
for some function $F(\cdot) \leq 1$.
If we fix $u \in A$ and integrate over all possible values of $\tau_{A \setminus u}$ while keeping $\tau_{B \cup C}$ and $\psi$ constant, we obtain \eqref{eq::Fhat}.

 To establish \eqref{eq::sigmaBC}, we multiply both denominator and enumerator of \eqref{eq::prob_condB_2} by $Q_{B \cup C, C} (g,\tau,\psi)$:
\[ \ba  & \P{\sigma_A=\tau_A | \sigma_B = \tau_B, G=g, \phi = \psi, ( \phi, \sigma) \in \Omega(n)} \\
&= (1 + o_n(1)) \frac{Q_{A \cup B, A \cup B} (g,\tau,\psi) Q_{B \cup C, C} (g,\tau,\psi) }{\sum_{\tau' \in \Omega_{B \cup C}(n)}
Q_{A \cup B, A \cup B} (g,\tau',\psi) Q_{B \cup C, C} (g,\tau',\psi) } \\
&=  (1 + o_n(1)) \frac{\P{G=g,\sigma =\tau | \phi = \psi, ( \phi, \sigma) \in \Omega(n)}}{\P{G=g,\sigma_{B \cup C}=\tau_{B \cup C} | \phi = \psi, ( \phi, \sigma) \in \Omega(n)}} \\
&= (1 + o_n(1)) \P{\sigma_A=\tau_A | \sigma_{B \cup C} = \tau_{B \cup C}, G=g, \phi = \psi, ( \phi, \sigma) \in \Omega(n)}.   \ea \]
Integrating again over $\tau_{A \setminus u}$ gives \eqref{eq::sigmaBC}.

 We use \eqref{eq::Fhat} to obtain
\be \ba &\P{\sigma_u=\tau_u | \sigma_B = \tau_B, G_{A \cup B}=g_{A \cup B}, \phi_{A \cup B} = \psi_{A \cup B}, ( \phi, \sigma) \in \Omega(n)} \\
&= \s{\widehat{g},\psi_C}{} \P{  \sigma_u=\tau_u |  \sigma_B = \tau_B, G=\widehat{g}, \phi = (\psi_{A \cup B},\psi_C), ( \phi, \sigma) \in \Omega(n)} \\ & \quad \quad  \quad \cdot \P{G = \widehat{g}, \phi_C = \psi_C | \sigma_B = \tau_B, G_{A \cup B}=g_{A \cup B}, \phi_{A \cup B} = \psi_{A \cup B}, ( \phi, \sigma) \in \Omega(n) } \\
&= (1 + o_n(1))  \widehat{F}\lr{g_{A \cup B},\tau_{u \cup B},\psi_{A \cup B}} +o_n(1) \\
&= (1 + o_n(1)) \P{\sigma_u=\tau_u | \sigma_B = \tau_B, G=g, \phi = \psi, ( \phi, \sigma) \in \Omega(n)} +o_n(1). \ea \label{eq::gABC}\ee
Combining \eqref{eq::sigmaBC} and \eqref{eq::gABC} gives \[ \ba &\P{\sigma_u=\tau_u | \sigma_{B \cup C} = \tau_{B \cup C}, G=g, \phi = \psi, ( \phi, \sigma) \in \Omega(n)} \\&= (1 + o_n(1)) \P{\sigma_u=\tau_u | \sigma_B = \tau_B, G=g, \phi = \psi, ( \phi, \sigma) \in \Omega(n)} \\ &= (1 + o_n(1)) \P{\sigma_u=\tau_u | \sigma_B = \tau_B, G_{A \cup B}=g_{A \cup B}, \phi_{A \cup B} = \psi_{A \cup B}, ( \phi, \sigma) \in \Omega(n)}, \ea \]
i.e., the claim \eqref{eq::CondOmega}. 

Our last step consists in removing the condition  $(\sigma,\phi) \in \Omega(n)$:
Put $\epsilon(n) = 1 - \P{(\sigma,\phi) \in \Omega(n)}$, then $\lim_{n \to \infty} \epsilon(n) = 0$. Indeed, $\sum_{u \in C} \indicator{\sigma_u = k} \phi_u = \sum_{u \in V} \indicator{\sigma_u = k} \phi_u + \bigO \lr{n^{17/72}}$, where the sum over $V$ has $n \frac{\PHI}{q}$ as a mean. The claim thus follows upon applying Hoeffding's inequality (the weights are assumed to be bounded by $n^{\alpha}$).

  Consider the random variable\\ $\P{( \phi, \sigma) \in \Omega(n)|  \sigma_B,  G_{A \cup B}, \phi_{A \cup B} } = \E{ 1_{(\phi, \sigma) \in \Omega(n)} | \sigma_B,  G_{A \cup B}, \phi_{A \cup B}} $. It has expectation $1 - \epsilon(n)$, so that 
\be\P{\E{ 1_{(\phi, \sigma) \in \Omega(n)} | \sigma_B,  G_{A \cup B}, \phi_{A \cup B}} \geq 1 - \sqrt{\epsilon(n)}} \geq 1 - 2 \sqrt{\epsilon(n)}. \label{eq::prob_exp}\ee
 Indeed, if contrary to our claim $f:= \E{ 1_{(\phi, \sigma) \in \Omega(n)} | \sigma_B,  G_{A \cup B}, \phi_{A \cup B}} \geq 1 - \sqrt{\epsilon(n)}$ with probability at most $1 - 2 \sqrt{\epsilon(n)}$, then
\[ \ba \E{f} &\leq 1 \cdot (1 - 2 \sqrt{\epsilon(n)}) + (1 -  \sqrt{\epsilon(n)})\cdot  2 \sqrt{\epsilon(n)} < 1 - \epsilon(n).
 \ea \]
Similarly, for $B \cup C$, 
\be\P{\E{ 1_{(\phi, \sigma) \in \Omega(n)} | \sigma_{B \cup C},  G, \phi} \geq 1 - \sqrt{\epsilon(n)}} \geq 1 - 2 \sqrt{\epsilon(n)}. \label{eq::prob_exp_BC} \ee
It follows that, with probability at least $1 - \bigO \lr{ \sqrt{\epsilon(n)} }$,
\[ \ba &\P{\sigma_u = +| \sigma_B ,  G_{A \cup B}, \phi_{A \cup B}} \\
&= \lr{ 1 -\bigO \lr{ \sqrt{\epsilon(n)} }}\P{\sigma_u = + | \sigma_B,  G_{A \cup B}, \phi_{A \cup B}, (\phi, \sigma) \in \Omega(n)} \\ & \quad+ \bigO \lr{ \sqrt{\epsilon(n)}} \P{\sigma_u = +| \sigma_B ,  G_{A \cup B}, \phi_{A \cup B}, (\phi, \sigma) \notin \Omega(n)} \\
&= (1 + o_n(1)) \P{\sigma_u = + | \sigma_{B \cup C} , G, \phi, (\phi, \sigma) \in \Omega(n)} + o_n(1)\\
&= (1 + o_n(1)) \P{\sigma_u = + | \sigma_{B \cup C} , G, \phi} + o_n(1),
\ea
\]
where we used \eqref{eq::prob_exp}, \eqref{eq::CondOmega} and \eqref{eq::prob_exp_BC} in the first, second, respectively last equality. 
\end{proof}
We are now in a position to proof Theorem \ref{thm::ConvergenceToHalf}:
\begin{proof}[Proof of Theorem \ref{thm::ConvergenceToHalf}]
Put $A =  G_{R - 1}$, $B = \partial G_R$ and $C = G \setminus G_R$. We use the monotonicity property of conditional variance \footnote{For random variables $X,Y,Z$, we have Var$(\E{X|Y}) \leq $ Var$(\E{X|Y,Z})$. Indeed, put $z = \E{X|Y,Z}$, then by Jensen's inequality $\E{z|Y}^2 \leq \E{z^2|Y}.$ So that, after taking expectations on both sides, $\E{\E{X|Y}^2} \leq \E{\E{X|Y,Z}^2}$. Writing out the definition of the variance then establishes the claim. } to obtain that, for any $s \in \spm$,
\[ 0 \leq \text{Var}( \E{\indicator{\sigma_u=s} | \sigma_v , G}) \leq \text{Var}(\E{\indicator{\sigma_u=s} | \sigma_{B \cup C} , G,\phi }) + o_n(1) \]
since $v \in B \cup C$ w.h.p. It suffices to show that the right-hand side tends to $0$, because this implies that  $\P{\sigma_u = s| \sigma_v , G} \overset{\mathbb{P}} \to 1/q $. 

To show that the right-hand side tends indeed to $0$, it suffices that \\$\P{\sigma_u = s | \sigma_{B \cup C} , G,\phi } \overset{\mathbb{P}} \to 1/q.$

Now, by using the partition $A \cup B \cup C$ of $V(G)$ in Lemma \ref{lm::MarkovField}, we have, since $G_R \leq n^{1/9}$ w.h.p., and all weights are bounded by $n^{\alpha}$ w.h.p. (this follows from a union bounded over all vertices), 
\[  \P{\sigma_u = s | \sigma_{B \cup C} , G,\phi} \overset{w.h.p.} = \P{\sigma_u = s |  \sigma_{\partial G_R} , G_R, \phi_{G_R}} + o_n(1) .   \]
Theorem \ref{thm::coupling} entails that the local neighbourhood is w.h.p. equal to $T^{\text{Poi}}$. Let $T^n$ be an independent copy of $T^{\text{Poi}}$ with root $\rho$, spins $\tau^n$ and weights $\psi^n$. Note that we stress the dependence on $n$, because the Poisson-tree is sampled again for each $n$.
\be \ba \P{\sigma_u = s |  \sigma_{\partial G_R} , G_R, \phi_{G_R}} + o_n(1) & \overset{w.h.p.} = \P{\tau^n_{\rho} = s |  \tau^n_{\partial T^n_R} , T^n_R, \psi_{T^n_R}} + o_n(1) \\
&= \P{\tau^n_{\rho} = s |  \tau^n_{\partial T^n_R} , T^n_R} + o_n(1), \label{eq::ComTree} \ea \ee
due to the coupling from Theorem \ref{thm::coupling}.
 By Theorem \ref{thm::main_broadcasting}, the right-hand side of \eqref{eq::ComTree} tends to $1/q$ in probability. 
\end{proof}

Using the following auxiliary lemma, Theorem \ref{thm::Main} follows  from Theorem \ref{thm::ConvergenceToHalf}:
\begin{lemma}
Assume that  $(a-b)^2 \PHItwo \leq q(a+b)$. Let $G$ be an observation of the DC-SBM, with true communities $\{\sigma_i\}_{i=1}^n$. Let $u$ and $v$ be two uniformly picked vertices. Let $\{\widehat{\sigma}_i\}_{i=1}^n$ be a reconstruction of the communities, based on the observation $G$. Assume that there exists $\delta > 0$ such that
\[ f(n):= \frac{1}{n} \s{i=1}{n} \indicator{ \sigma_i = \widehat{\sigma}_i} \geq \frac{1}{q} + \delta, \]
with high probability. Then, there exists $s \in \spm$, such that 
 $ \P{\sigma_u = s | \sigma_v, G}$ does \emph{not} converge in probability to $1/q$.
\label{lm::NoConvergence}
\end{lemma}

\begin{proof}
Assume for a contradiction that for every $s$, $ \P{\sigma_u = s | \sigma_v, G}$ tends to $1/q$ in probability. Since $\widehat{\sigma}_u$ depends on $\sigma_u$ only through $G$, we have for any $s \in \spm$, 
\BA \text{Var}\lr{\E{\indicator{\sigma_u = s}|\sigma_v,G}} 
&= \text{Var}\lr{\E{\indicator{\sigma_u = s}|\widehat{\sigma}_u,\sigma_v,G}} \\
& \geq   \text{Var} \lr{\E{\indicator{\sigma_u = s}|\widehat{\sigma}_u}},\label{eq::viol} \EA
 where the term on the left tends to zero by assumption. 
By definition of $f(n)$,
\[\begin{aligned}
1/q + \delta +o(1) &\leq \sum_s \P{\sigma_u = \widehat{\sigma}_u|\widehat{\sigma}_u=s} \P{\widehat{\sigma}_u=s}.
\end{aligned}  \] 
Hence, for large enough $n$, there must be an $s$ such that $\P{\sigma_u = \widehat{\sigma}_u|\widehat{\sigma}_u=s} \geq 1/q + \delta/2$ and $\P{\widehat{\sigma}_u=s} \geq \frac{\delta}{3q}$. As a consequence, the term on the right of \eqref{eq::viol} does not tend to zero.  
\end{proof}

 We summarize these results in Theorem \ref{thm::Main}:
\begin{proof}[Proof of Theorem \ref{thm::Main}]
Combine Theorem \ref{thm::ConvergenceToHalf} and Lemma \ref{lm::NoConvergence}.
\end{proof}

\section{Acknowledgement}
The authors would like to thank Joe Neeman for an inspiring discussion.

\bibliographystyle{plain}
\bibliography{literature}
\end{document}